\documentclass[11pt]{article}

\usepackage{amssymb,amsmath,amsfonts,amsthm}
\usepackage{latexsym}
\usepackage{graphics}
\usepackage{indentfirst}

\newtheorem*{thm*}{Theorem}
\newtheorem{thm}{Theorem}
\newtheorem{lem}[thm]{Lemma}

\newtheorem{cor}[thm]{Corollary}
\newtheorem{conj}[thm]{Conjecture}

\newtheorem{ques}[thm]{Question}

\newcommand{\N}{\mathbb{N}}

\begin{document}

\title{A Note on the Equitable Choosability of Complete Bipartite Graphs}

\author{Jeffrey A. Mudrock\footnotemark[1], Madelynn Chase\footnotemark[1], Isaac Kadera\footnotemark[1], Ezekiel Thornburgh\footnotemark[1], and Tim Wagstrom\footnotemark[1]}

\footnotetext[1]{Department of Mathematics, College of Lake County, Grayslake, IL 60030.  E-mail:  {\tt {jmudrock@clcillinois.edu}}}

\date{2018}

\maketitle

\begin{abstract}

In 2003 Kostochka, Pelsmajer, and West introduced a list analogue of equitable coloring called equitable choosability.  A $k$-assignment, $L$, for a graph $G$ assigns a list, $L(v)$, of $k$ available colors to each $v \in V(G)$, and an equitable $L$-coloring of $G$ is a proper coloring, $f$, of $G$ such that $f(v) \in L(v)$ for each $v \in V(G)$ and each color class of $f$ has size at most $\lceil |V(G)|/k \rceil$.  Graph $G$ is said to be equitably $k$-choosable if an equitable $L$-coloring of $G$ exists whenever $L$ is a $k$-assignment for $G$.  In this note we study the equitable choosability of complete bipartite graphs.  A result of Kostochka, Pelsmajer, and West implies $K_{n,m}$ is equitably $k$-choosable if $k \geq \max \{n,m\}$ provided $K_{n,m} \neq K_{2l+1, 2l+1}$.  We prove $K_{n,m}$ is equitably $k$-choosable if $m \leq \left\lceil (m+n)/k \right \rceil(k-n)$ which gives $K_{n,m}$ is equitably $k$-choosable for certain $k$ satisfying $k < \max \{n,m\}$.  We also give a complete characterization of the equitable choosability of complete bipartite graphs that have a partite set of size at most 2.

\medskip

\noindent {\bf Keywords.} graph coloring, equitable coloring, list coloring, equitable choosability.

\noindent \textbf{Mathematics Subject Classification.} 05C15.

\end{abstract}

\section{Introduction}\label{intro}

All graphs in this note are assumed to be finite, simple graphs unless otherwise noted.  We use $K_{n,m}$ for the equivalence class of complete bipartite graphs with partite sets of size $n$ and $m$ where $n$ and $m$ are always positive integers.  Generally speaking we follow West~\cite{W01} for basic terminology and notation.  The focus of this note is the equitable choosability of complete bipartite graphs.  Equitable choosability is a list analogue of equitable coloring that was introduced by Kostochka, Pelsmajer, and West in 2003~\cite{KP03}.  Before introducing equitable choosability we quickly review equitable coloring and list coloring.

\subsection{Equitable Coloring and List Coloring}

The notion of equitable coloring was formally introduced by Meyer in 1973~\cite{M73}, but the study of equitable coloring began with a conjecture of Erd\H{o}s in 1964~\cite{E64} (see Theorem~\ref{thm: HS} below).  A proper $k$-coloring, $f$, of a graph $G$ is said to be an \emph{equitable $k$-coloring} if the $k$ color classes associated with $f$ differ in size by at most 1.  If $f$ is an equitable $k$-coloring of the graph $G$, it is easy to see that the size of each color class associated with $f$ must be $\lceil |V(G)|/k \rceil$ or $\lfloor |V(G)|/k \rfloor$.  We say that a graph $G$ is \emph{equitably $k$-colorable} if there exists an equitable $k$-coloring of $G$.  From an applied perspective, equitable colorings are useful when we wish to find a proper coloring of a graph without over or under-using any colors (see~\cite{JR02}, \cite{KJ06}, \cite{P01}, and~\cite{T73} for examples of applications). 

Unlike the typical vertex coloring problem, if a graph is equitably $k$-colorable, it need not be equitably $(k+1)$-colorable.  Indeed, $K_{2m+1,2m+1}$ is equitably $k$-colorable for each even $k$ less than $2m+1$, it is not equitably $(2m+1)$-colorable, and it is equitably $k$-colorable for each $k \geq 2m+2 = \Delta(K_{2m+1,2m+1})+1$ where we use $\Delta(G)$ to denote the largest degree of a vertex in $G$ (see~\cite{LW96} for further details).

In 1970 Hajn\'{a}l and Szemer\'{e}di proved Erd\H{o}s' 1964 conjecture.  In particular, they proved the following. 

\begin{thm}[\cite{HS70}] \label{thm: HS}
Every graph $G$ has an equitable $k$-coloring when $k \geq \Delta(G)+1$.
\end{thm}

In 1994 Chen, Lih, and Wu~\cite{CL94} conjectured that the result of Theorem~\ref{thm: HS} can be improved by 1 for most connected graphs.  Their conjecture is still open and is known as the $\Delta$-Equitable Coloring Conjecture ($\Delta$-ECC for short).   

\begin{conj}[\cite{CL94}, {\bf $\Delta$-ECC}] \label{conj: ECC}
A connected graph $G$ is equitably $\Delta(G)$-colorable if it is different from $K_m$, $C_{2m+1}$, and $K_{2m+1,2m+1}$.
\end{conj}

Conjecture~\ref{conj: ECC} has been proven true for interval graphs, bipartite graphs, outerplanar graphs, subcubic graphs, certain planar graphs, and several other classes of graphs (see~\cite{CL94}, \cite{DW18}, \cite{DZ18}, \cite{L98}, \cite{LW96} and~\cite{YZ97}).

List coloring is another variation on the classic vertex problem introduced independently by Vizing~\cite{V76} and Erd\H{o}s, Rubin, and Taylor~\cite{ET79} in the 1970's.  For list coloring we associate with a graph $G$ a \emph{list assignment}, $L$, that assigns to each vertex $v \in V(G)$ a list, $L(v)$, of available colors.  Graph $G$ is said to be \emph{$L$-colorable} if there exists a proper coloring $f$ of $G$ such that $f(v) \in L(v)$ for each $v \in V(G)$ (we refer to $f$ as a \emph{proper $L$-coloring} of $G$).  A list assignment $L$ is called a \emph{k-assignment} for $G$ if $|L(v)|=k$ for each $v \in V(G)$.  We say $G$ is \emph{k-choosable} if $G$ is $L$-colorable whenever $L$ is a $k$-assignment for $G$.   

\subsection{Equitable Choosability}

In 2003 Kostochka, Pelsmajer, and West introduced a list analogue of equitable coloring called equitable choosability~\cite{KP03}.  They use the word equitable to capture the idea that no color may be used excessively often.  Specifically, if $L$ is a $k$-assignment for the graph $G$, a proper $L$-coloring of $G$ is \emph{equitable} if each color appears on at most $\lceil |V(G)|/k \rceil$ vertices.  Such a coloring is called an \emph{equitable $L$-coloring} of $G$, and we call $G$ \emph{equitably $L$-colorable} when an equitable $L$-coloring of $G$ exists.  We say $G$ is \emph{equitably $k$-choosable} if $G$ is equitably $L$-colorable whenever $L$ is a $k$-assignment for $G$.  So, the upper bound on the number of times we are allowed to use a color in an equitable $L$-coloring is the same as the upper bound on the size of the color classes in an ordinary equitable $k$-coloring.  It is conjectured in~\cite{KP03} that Theorem~\ref{thm: HS} and the $\Delta$-ECC hold in the list context. 

\begin{conj}[\cite{KP03}] \label{conj: KPW1}
Every graph $G$ is equitably $k$-choosable when $k \geq \Delta(G)+1$.
\end{conj}

\begin{conj}[\cite{KP03}] \label{conj: KPW2}
A connected graph $G$ is equitably $k$-choosable for each $k \geq \Delta(G)$ if it is different from $K_m$, $C_{2m+1}$, and $K_{2m+1,2m+1}$.
\end{conj}

In~\cite{KP03} it is shown that Conjectures~\ref{conj: KPW1} and~\ref{conj: KPW2} hold for forests, complete bipartite graphs, connected interval graphs, and 2-degenerate graphs with maximum degree at least 5.  Conjectures~\ref{conj: KPW1} and~\ref{conj: KPW2} have also been verified for outerplanar graphs~\cite{ZB10}, series-parallel graphs~\cite{ZW11}, graphs with small maximum average degree~\cite{DZ18}, powers of cycles~\cite{KM18}, and certain planar graphs (see~\cite{DW18}, \cite{LB09}, \cite{ZB08}, and~\cite{ZB15}).  In 2013, Kierstead and Kostochka made substantial progress on Conjecture~\ref{conj: KPW1}, and proved it for all graphs of maximum degree at most 7 (see~\cite{KK13}).

Most of the research on equitable choosability has been focused on Conjectures~\ref{conj: KPW1} and~\ref{conj: KPW2}.  There is not much research that considers the equitable $k$-choosability of a graph $G$ when $k < \Delta(G)$.  In~\cite{KP03} it is shown that if $G$ is a forest and $k \geq 1 + \Delta(G)/2$, then $G$ is equitably $k$-choosable.  It is also shown that this bound is tight for forests.  Also, in~\cite{KM18}, it is conjectured that if $T$ is a total graph, then $T$ is equitably $k$-choosable for each $k \geq \max \{\chi_\ell(T), \Delta(T)/2 + 2 \}$ where $\chi_\ell(T)$ is the smallest $m$ such that $T$ is $m$-choosable.  In this note we will present some results on the equitable choosability of complete bipartite graphs that will give us equitable $k$-choosability for values of $k$ that are smaller than the maximum degree of the graph. 

Most results about equitable choosability state that some family of graphs is equitably $k$-choosable for all $k$ above some constant; even though, as with equitable coloring, if $G$ is equitably $k$-choosable, it need not be equitably $(k+1)$-choosable.  It is rare to have a result that determines whether a family of graphs is equitably $k$-choosable for each $k \in \N$.  In this note we have results of this form: we will completely determine when $K_{1,m}$ and $K_{2,m}$ are equitably $k$-choosable.  It is worth mentioning that a new list analogue of equitable coloring called proportional choosability was recently introduced in~\cite{KM182}, and a simple characterization of the proportional choosability of stars (i.e. graphs of the form $K_{1,m}$) has been found~\cite{KM182}.     

\subsection{An Open Question and Outline}

We now present a brief outline of our results.  The following open question motivated our research.

\begin{ques} \label{ques: main}
For what values of $k$ is the complete bipartite graph $K_{n,m}$ equitably $k$-choosable?
\end{ques}
   
Since Conjecture~\ref{conj: KPW2} is known to be true for complete bipartite graphs, we know that when $n \neq m$ or $n$ is even, $K_{n,m}$ is equitably $k$-choosable if $k \geq \max \{n,m\}$.  Our first two results give a partial answer to Question~\ref{ques: main}.  

\begin{thm} \label{thm: main}
$K_{n,m}$ is equitably $k$-choosable if
$$ m \leq \left\lceil\frac{m+n}{k} \right \rceil(k-n).$$
Consequently, for each $i=2, 3, \ldots, \left \lceil \sqrt{1 + m/n} \right \rceil $, $K_{n,m}$ is equitably $k$-choosable if
$$ \frac{m + in}{i} \leq k < \frac{m+n}{i-1}.$$
\end{thm}

\begin{thm} \label{thm: badk}
$K_{n,m}$ is not equitably $k$-choosable if
$$ m > \left\lceil\frac{m+n}{k} \right \rceil(k-1).$$
Consequently, for each $i=n+1, n+2, \ldots, n+m $, $K_{n,m}$ is not equitably $k$-choosable if
$$ \frac{m + n}{i} \leq k < \frac{m+i}{i}.$$
\end{thm}

Note Theorem~\ref{thm: main} is only interesting if $k<m+n$ (i.e. if $\lceil (m+n)/k \rceil \geq 2$).  So, the inequality in Theorem~\ref{thm: main} is easier to satisfy if $ m=\max\{m,n\}$.  Since $K_{m,n}=K_{n,m}$, it is more helpful to apply Theorem~\ref{thm: main} when $m \geq n$.  Also, Theorems~\ref{thm: main} and~\ref{thm: badk} do not address all possible values of $k$ when $n \geq 2$.  However, Theorems~\ref{thm: main} and~\ref{thm: badk} do address all possible $k$ values when $n=1$.  In particular, we have the following corollary.

\begin{cor} \label{cor: star}
$K_{1,m}$ is equitably $k$-choosable if and only if 
$$ m \leq \left\lceil\frac{m+1}{k} \right \rceil(k-1).$$
\end{cor}

Corollary~\ref{cor: star} answers Question~\ref{ques: main} for stars (i.e. the $n=1$ case).  Notice that the result in~\cite{KP03} for forests only implies that $K_{1,m}$ is equitably $k$-choosable whenever $k \geq 1 + m/2$.  Using further ideas, we also prove that Theorem~\ref{thm: badk} gives the best possible result for $K_{2,m}$ and hence answer Question~\ref{ques: main} in the case that $n=2$.  In particular, we prove the following result.

\begin{thm} \label{thm: K2m}
$K_{2,m}$ is equitably $k$-choosable if and only if 
$$ m \leq \left \lceil \frac {m+2}{k}\right \rceil (k-1).$$
\end{thm}

We now quickly present some illustrative examples of Corollary~\ref{cor: star} and Theorem~\ref{thm: K2m}.  For example, note Corollary~\ref{cor: star} implies that $K_{1,25}$ is equitably $k$-choosable if and only if $ k \in \{6, 8, 10, 11, 12 \} \cup \{z \in \N : z \geq 14 \}.$  Notice that the result in~\cite{KP03} for forests only implies that $K_{1,25}$ is equitably $k$-choosable whenever $k \geq 14$.  Similarly, note that Theorem~\ref{thm: K2m} implies that $K_{2,139}$ is equitably $k$-choosable if and only if $k \in \{14, 15, 17, 19, 20, 21, 22, 23\} \cup \{z \in \N: z \geq 25\}.$  

With Corollary~\ref{cor: star} and Theorem~\ref{thm: K2m} in mind, one might conjecture that $K_{3,m}$ is equitably $k$-choosable if and only if $ m \leq \left \lceil \frac {m+3}{k}\right \rceil (k-1)$ (i.e. Theorem~\ref{thm: badk} gives the best possible result in the case that $n=3$).  However, this is not true.  Indeed, $3 \leq \lceil (3+3)/2 \rceil (2-1)$ and $4 \leq \lceil (3+4)/2 \rceil (2-1)$, yet it easy to see that both $K_{3,3}$ and $K_{3,4}$ are not equitably 2-choosable since such graphs are not even 2-choosable (see~\cite{ET79} and~\cite{V76}).

\section{Proofs of Results} \label{results}

We begin with a useful lemma.

\begin{lem} \label{lem: construct} 
Suppose that $G = \overline{K_m}$ \footnote{$\overline{K_m}$ denotes the complement of a complete graph on $m$ vertices.  So, $G$ consists of $m$ isolated vertices} and $L^{(1)}$ is a list assignment for $G$ such that $|L^{(1)}(v)| \geq \eta$ for each $v \in V(G)$. If $\sigma \in \mathbb{N}$ is such that $m \leq \sigma \eta$, then there is a proper $L^{(1)}$-coloring of $G$ that uses no color more than $\sigma$ times.
\end{lem}

\begin{proof}
We begin by describing an inductive process for coloring $G$. If there is no color in at least $\sigma$ of the lists associated with $L^{(1)}$ the process stops. Otherwise, there is a color, $c_1$, in at least $\sigma$ of the lists associated with $L^{(1)}$, and we arbitrarily color $\sigma$ of the vertices that have $c_1$ in their list with $c_1$. Call the set of all vertices colored with $c_1$, $A_1$. 

Now, we inductively continue in this fashion. In particular, for $t \geq 2$, let $L^{(t)}(v) = L^{(t-1)}(v) - \{c_{t-1}\}$ for each $v \in V(G) - \bigcup_{i=1}^{t-1} A_i$. Note that $|L^{(t)}(v)| \geq \eta -(t-1)$ for each $v \in V(G)- \bigcup_{i=1}^{t-1} A_i$. If there is no color in at least $\sigma$ of the lists associated with $L^{(t)}$ the process stops. Otherwise, there is a color, $c_t$, in at least $\sigma$ of the lists associated with $L^{(t)}$, and we arbitrarily color $\sigma$ of the uncolored vertices that have $c_t$ in their list with $c_t$. Then, call the set of all vertices colored with $c_t$, $A_t$. 

Now, if the process stops at some $t \leq \eta$, we can complete a proper $L^{(1)}$-coloring of $G$ with the property that no color is used more than $\sigma$ times by greedily coloring each $v \in V(G)- \bigcup_{i=1}^{t-1} A_i$ with a color in $L^{(t)}(v)$; this is possible since $|L^{(t)}(v)| \geq 1$ whenever $t \leq \eta$. Otherwise, we get to $t = \eta$ and then $\sigma$ vertices in $ V(G) - \bigcup_{i=1}^{\eta -1} A_i$ are colored with $c_{\eta}$. After coloring vertices with $c_{\eta}$, $\sigma \eta$ of the vertices in $V(G)$ are colored. Since $|V(G)|=m \leq \sigma \eta$, we must have colored all the vertices in $V(G)$ which means we have obtained a proper $L^{(1)}$-coloring of $G$ such that no color is used more than $\sigma$ times. 
\end{proof}

We are now ready to prove Theorem~\ref{thm: main}.

\begin{proof}
Assume $m,n,k \in \N$ satisfy $m \leq \lceil (m+n)/k \rceil (k-n)$. Let $G$ be a copy of $K_{n,m}$ with partite sets $\{u_{1}, u_{2}, \dots, u_{n}\}$ and $A=\{v_{1}, v_{2}, \dots, v_{m}\}$. Let $L$ be a $k$-assignment for $G$. We will show that $G$ is equitably $L$-colorable.  

We must have $k > n$.  So, there exists a $c_i \in L(u_i)$ for each $i \in \{1,2,\ldots,n\}$ such that $c_1,c_2,\ldots,c_n$ are pairwise distinct. We color $u_i$ with $c_i$ for each $i \in \{1,2,\ldots,n\}$. Now, for each $v \in A$, let $L'(v) = L(v)-\{c_1,c_2,\dots,c_n\}$.  Clearly $|L(v)| \geq k-n$. Lemma~\ref{lem: construct} implies that we can find a proper $L'$-coloring of $G[A]$ such that no color is used more than $\lceil (m+n)/k \rceil$ times. Such a coloring completes an equitable $L$-coloring of $G$.
\end{proof}

Now, we prove Theorem~\ref{thm: badk}

\begin{proof}

The result is obvious when $k=1$.  So, assume $k \geq 2$ and $n,m \in \N$ satisfy $m > \lceil (m+n)/k \rceil (k-1)$. Let $G$ be a copy of $K_{n,m}$ with partite sets $A'=\{u_{1}, u_{2}, \ldots, u_n\}$ and $A=\{v_{1}, v_{2}, \dots, v_{m}\}$.  Next, let $L$ be the $k$-assignment for $G$ that assigns $\{1, 2, \dots, k\}$ to every vertex in $V(G)$.  It suffices to show $G$ is not equitably $L$-colorable.  For the sake of contradiction, assume $f$ is an equitable $L$-coloring of $G$. Suppose $|f(A')|=a$.  Clearly $1 \leq a < k$, and without loss of generality, we may assume $f(A')= \{1, 2, \ldots, a \}$. Since $f$ is a proper coloring, this means $f(A) \subseteq \{a+1, \dots, k\}$ and $|f^{-1}(\{a+1, \dots,k\})|= m$. Also, since $f$ is an equitable $L$-coloring, $| f^{-1}(j)|  \leq \lceil |V(G)|/k \rceil = \lceil (m+n)/k \rceil$ when $j \in \{a+1, \dots, k\}$. Thus, 
$$m= \sum^{k}_{i=a+1} | f^{-1}(i) | \leq \sum^{k}_{j=a+1} \left \lceil \frac{m+n}{k} \right \rceil = (k-a)\left \lceil \frac{m+n}{k} \right \rceil \leq (k-1)\left \lceil \frac{m+n}{k} \right \rceil$$ 
which is a contradiction. 
\end{proof}

It is clear that Corollary~\ref{cor: star} follows immediately from Theorems~\ref{thm: main} and~\ref{thm: badk}.  

We now turn our attention to proving Theorem~\ref{thm: K2m}.  We begin by proving a lemma that will allow us to restrict our attention to the case where the lists corresponding to the vertices in the partite set of size two are disjoint.

\begin{lem} \label{lem: easycase}
Suppose $G =K_{2,m}$ and the partite sets of $G$ are $A' = \{u_1,u_2\}$ and $A = \{v_1,v_2,\ldots, v_m\}$. Also, suppose that $L$ is a $k$-assignment for $G$ such that $L(u_1) \cap L(u_2) \neq \emptyset$. If $m \leq \lceil (m+2)/k \rceil (k-1)$ and $k <m+2$, then $G$ is equitably $L$-colorable.
\end{lem}

\begin{proof}
Suppose $z_1 \in L(u_1) \cap L(u_2)$. Also, suppose that $m \leq \lceil (m+2)/k \rceil(k-1)$ and $k < m+2$.  We color $u_1$ and $u_2$ with $z_1$. For each $v \in A$ we let $L'(v) = L(v) - \{z_1\}$ which implies $|L'(v)| \geq k-1$. Lemma~\ref{lem: construct} implies that we can find a proper $L'$-coloring of $G[A]$ such that no color is used more than $ \lceil (m+2)/k \rceil$ times. Such a coloring completes an equitable $L$-coloring of $G$.
\end{proof}

We are now ready to focus on the case where the lists corresponding to the vertices in the partite set of size two are disjoint.  The next lemma shows that in this case there exists a way to color the vertices in the partite set of size two such that both used colors do not appear in more than one fourth of the lists corresponding to the vertices in the other partite set.

\begin {lem} \label{lem: badlists}
Suppose $G = K_{2,m}$ and the partite sets of $G$ are $A' = \{u_1, u_2\}$ and $A = \{v_1, v_2,\cdots , v_m\}$. Also, suppose that $L$ is a $k$-assignment for $G$ such that $L(u_1) \cap L(u_2) = \emptyset$. There must exist a $c_q \in L(u_1)$ and $c_r \in L(u_2)$ such that $|\{v\in A:\{ c_q,c_r\} \subseteq L(v)\}| \leq m/4$.
\end{lem} 

\begin{proof}
Suppose $L(u_1) = \{c_1, \ldots , c_k\}$ and $L(u_2)= \{c_{k+1}, \ldots ,c_{2k}\}$. Let $P = L(u_1) \times L(u_2)$. Then, for each $(c_i,c_j) \in P$ we say $(c_i,c_j)$ is \emph{contained in} $L(v)$ if $\{c_i,c_j\} \subseteq L(v)$. For each $(c_i,c_j) \in P$, we let $\beta_{i,j} = |\{v\in A : (c_i,c_j) \text{ is contained in } L(v)\}|$.  

For each $v \in A$, we let $\gamma_v = |\{(c_i,c_j) \in P : (c_i,c_j) \text { is contained in } L(v)\}|$.  From these definitions, it is easy to see that $\sum _{(c_i,c_j) \in P} \beta _{i,j} = \sum _{i=1}^{m} \gamma _{v_i}$.  If $v \in A$, $a = |L(u_1) \cap L(v)|$, and $b = |L(u_2) \cap L(v)|$, then the number of elements of $P$ contained in $L(v)$ is $ab$ and $a+b \leq k$. So, $ab \leq \lceil k/2 \rceil \lfloor k/2 \rfloor$.  This means $\gamma _{v} \leq \lceil k/2 \rceil \lfloor k/2 \rfloor$ for each $v \in A$, and we have the following: 
$$\sum _{(c_i,c_j) \in P} \beta _{i,j} = \sum _{i=1}^{m} \gamma _{v_i} \leq \left \lceil \frac{k}{2} \right \rceil \left \lfloor \frac {k}{2} \right \rfloor m \leq \frac{k^2m}{4}.$$
  
For the sake of contradiction, suppose that there does not exist a $(c_i,c_j) \in P$ such that $\beta _{i,j} \leq m/4$. This means that 
$$\sum _{(c_i,c_j) \in P} \beta _{i,j} > \frac{k^2m}{4}$$
which is a contradiction.
\end{proof}

We are now ready to prove Theorem~\ref{thm: K2m}.

\begin{proof}
Note that the only if direction is implied by Theorem~\ref{thm: badk}. So, we prove the if direction. The result is obvious when $k \geq m+2$ and when $k=1$. Notice that when $k=2$ the only values that satisfy the inequality are $m=1,2,3$ and it is easy to verify that $K_{2,1},K_{2,2},K_{2,3}$ are equitably 2-choosable. So, we may assume that $3 \leq k < m+2$.

Suppose $G = K_{2,m}$ and the partite sets of $G$ are $A' = \{u_1, u_2\}$ and $A = \{v_1, v_2,\cdots , v_m\}$. Also, suppose that $L$ is an arbitrary $k$-assignment for $G$ such that $m \leq \lceil (m+2)/k \rceil (k-1) $. Showing that an equitable $L$-coloring of $G$ exists will complete the proof. By Lemma~\ref{lem: easycase} we may assume that $L(u_1) \cap L(u_2) = \emptyset$. By Lemma~\ref{lem: badlists} we know there exists a $z_1 \in L(u_1)$ and $z_2\in L(u_2)$ such that $|\{v\in A:\{ z_1,z_2\} \subseteq L(v)\}| \leq m/4$. Color $u_1$ with $z_1$ and $u_2$ with $z_2$. Let $L^{(1)}(v) = L(v) - \{z_1,z_2\}$ for each $v \in A$. We will now construct a proper $L^{(1)}$-coloring of $G[A]$ that uses no color more than $ \lceil (m+2)/k \rceil$ times. Such a coloring will complete an equitable $L$-coloring of $G$. 

Note that $|L^{(1)}(v)| \geq k-2$ for each $v \in A$. Let $B_1 = \{ v\in A : |L^{(1)}(v)| = k-2\}$. By our choice of $z_1$ and $z_2$, $|B_1| \leq m/4$. We color a subset of the vertices in $B_1$ by the following inductive process: If there is no color appearing in at least $ \lceil (m+2)/k \rceil $ of the lists assigned by $L^{(1)}$ to vertices in $B_1$, then the process stops. Otherwise, there is a color, $c_1$, in at least $\lceil (m+2)/k \rceil $ of the lists assigned by $L^{(1)}$ to vertices in $B_1$, and we arbitrarily color $\lceil (m+2)/k \rceil$ of the vertices in $B_1$ that have $c_1$ in their list with $c_1$. Let $A_1$ be the set of vertices colored with $c_1$.

Now for $t \geq 2$ let $L^{(t)}(v) = L^{(t-1)}(v) - \{c_{t-1}\}$ for each $v \in A - \bigcup^{t-1}_{i=1} A_i$. Clearly, for each $v \in A - \bigcup^{t-1}_{i=1} A_i$, $|L^{(t)}(v)| \geq k-2-(t-1)$. Let $ B_t = \{ v \in A - \bigcup^{t-1}_{i=1} A_i: |L^{(t)}(v)| = k-2-(t-1)\}$. Suppose $w \in B_t$. Note that $|L^{(t-1)}(w)| \geq k-2-(t-2)$, and $L^{(t)} (w) = L^{(t-1)}(w) - \{c_t\}$. Since $|L^{(t)}(w)| = k-2-(t-1)$, $c_t \in L^{(t-1)}(w)$ and $|L^{(t-1)}(w)| = k-2-(t-2)$. This means $w\in B_{t-1}$ and $B_t \subseteq B_{t-1}$. If there are no colors in $\lceil (m+2)/k \rceil$ of the lists assigned by $L^{(t)}$ to vertices in $B_t$, then the process stops. Otherwise, there is a color $c_t$ in $\lceil (m+2)/k \rceil$ of those lists and we arbitrarily color $\lceil (m+2)/k \rceil$ vertices in $B_t$ that have $c_t$ in their list with $c_t$. Then, we let $A_t$ be the set of vertices colored with $c_t$. 

Suppose $t$ gets to $k-2$ but the process does not stop. This implies there is a color $c_{k-2}$ in at least $\lceil (m+2)/k \rceil$ of the lists assigned by $L^{(k-2)}$ to vertices in $B_{k-2}$. We have $B_{k-2} \subseteq B_1$ and $\bigcup ^{k-3} _{i=1} A_i \subseteq B_1$. Thus $|B_1| \geq (k-2)\lceil (m+2)/k \rceil$. Since $|B_1| \leq m/4$, $(k-2)\lceil (m+2)/k \rceil \leq m/4$. Since $k\geq 3$, $(k-2) > \frac{1}{4} (k-1)$. This implies 
$$ (k-2)\left \lceil \frac {m+2}{k}\right \rceil > \frac{1}{4} (k-1)\left \lceil \frac {m+2}{k}\right \rceil \geq \frac{m}{4}$$ which is a contradiction. So, the process must stop at some $t \leq k-2$.

Suppose the process stops at $t=\alpha$.  This means $\alpha \leq k-2$. Suppose there is no color in at least $\lceil (m+2)/k \rceil$ of the lists assigned by $L^{(\alpha)}$ to the vertices in $ A - \bigcup^{\alpha -1}_{i=1} A_i$. In this case we can complete an equitable $L$-coloring of $G$ by greedily coloring each $v \in A - \bigcup^{\alpha -1}_{i=1} A_i$ with a color in $L^{(\alpha)}(v)$. This is possible since for each $v \in A - \bigcup^{\alpha -1}_{i=1} A_i$,  $|L^{(\alpha)}(v)| \geq k-2-(\alpha -1) \geq k-2-(k-2-1) \geq 1$.

So, we may assume there is a color $c_\alpha$ that appears in at least $\lceil (m+2)/k \rceil$ of the lists assigned by $L^{(\alpha)}$ to the vertices in $ A - \bigcup^{\alpha -1}_{i=1} A_i$. Let $C_\alpha = \{v \in B_\alpha: c_\alpha \in L^{(\alpha)}(v)\}$. Since the process stopped at $t= \alpha$, we know $|C_\alpha| < \lceil (m+2)/k \rceil$. Color each vertex in $C_\alpha$ with $c_\alpha$. Then arbitrarily color $\lceil (m+2)/k \rceil - |C_\alpha|$ vertices in $A -\left(\left(\bigcup^{\alpha-1}_{i=1} A_i\right)\bigcup C_\alpha \right)$ that have $c_\alpha$ in their list with $c_\alpha$. Let $A_\alpha$ be the set of vertices colored with $c_\alpha$. For each $v \in A-\bigcup^{\alpha}_{i=1} A_i$, let $L^{(\alpha +1)}(v) = L^{(\alpha)}(v) - \{c_\alpha \}$. Note that if $v \in B_\alpha - C_\alpha$ then $c_\alpha \notin L^{(\alpha)}(v)$. Thus, for each $v \in A- \bigcup^{\alpha}_{i=1} A_i$, $|L^{(\alpha + 1)}(v)| \geq k-1-\alpha$. Note $|A -\bigcup^{\alpha}_{i=1} A_i| = m - \alpha \lceil (m+2)/k \rceil$. We know $m \leq \lceil (m+2)/k \rceil (k-1)$ which implies $m - \alpha \lceil (m+2)/k \rceil \leq \lceil (m+2)/k \rceil (k-1 -\alpha)$. Lemma~\ref{lem: construct} implies there is a proper $L^{(\alpha +1)}$-coloring of $G[A-\bigcup^{\alpha}_{i=1} A_i]$ that uses no color more than $\lceil (m+2)/k \rceil$ times. This completes a proper $L^{(1)}$-coloring of $G[A]$ that uses no color more than $\lceil (m+2)/k \rceil$ times.
\end{proof}

\vspace{5mm}

\noindent \textbf{Acknoledgement:}  The authors would like to thank Hemanshu Kaul and Michael Pelsmajer for their helpful comments on this note.  The authors would also like to thank Martin Maillard for his helpful comments that improved the readability of this note.

\end{document}